\documentclass[12pt]{amsart}
\usepackage[margin=1in]{geometry}
\usepackage{amssymb,amsfonts,amsmath,mathtools}
\usepackage{color}
\usepackage{soul}

\usepackage{enumerate}
\usepackage{mathrsfs}
\usepackage[breaklinks=true,colorlinks=true,linkcolor=blue,citecolor=blue,filecolor=blue,urlcolor=blue]{hyperref}
\hypersetup{linktocpage}
\usepackage[hiresbb]{graphicx,xcolor}
\usepackage[capitalise]{cleveref}

\usepackage{constants}
\newconstantfamily{abcon}{symbol=c}
\newconstantfamily{abcon2}{symbol=C}
\usepackage{bbm}

\newtheorem{theorem}{Theorem}[section]
\newtheorem{corollary}[theorem]{Corollary}
\newtheorem{proposition}[theorem]{Proposition}
\newtheorem{lemma}[theorem]{Lemma}

\newtheorem{question*}{Question}
\newtheorem{problem*}{Problem}

\theoremstyle{definition}

\theoremstyle{remark}
\newtheorem*{remark}{Remark}

\numberwithin{equation}{section}

\crefname{figure}{Figure}{Figures}
\theoremstyle{plain}
\newtheorem*{theorem*}{Theorem}
\crefname{theorems}{Theorem}{Theorems}
\crefname{corollaries}{Corollary}{Corollaries}
\newtheorem*{corollary*}{Corollary}
\crefname{corollaries*}{Corollary}{Corollaries}
\crefname{lemma}{Lemma}{Lemmata}
\crefname{proposition}{Proposition}{Propositions}
\crefname{conjectures}{Conjecture}{Conjectures}
\newtheorem*{conjonjecture*}{Conjecture}
\crefname{conjonjectures*}{Conjecture}{Conjectures}
\crefname{definitions}{Definition}{Definitions}
\crefname{hypotheses}{Hypothesis}{Hypotheses}

\newcommand{\R}{\mathbb{R}}
\newcommand{\Q}{\mathbb{Q}}

\newcommand{\re}{\textup{Re}}
\newcommand{\im}{\textup{Im}}

\newcommand{\GL}{\mathrm{GL}}

\renewcommand{\tilde}{\widetilde}

\renewcommand{\bar}{\overline}

\renewcommand{\epsilon}{\varepsilon}

\renewcommand{\Im}{\mathrm{Im}}

\renewcommand{\pmod}[1]{\, (\mathrm{mod} {\, #1})}

\renewcommand{\Re}{\mathrm{Re}}

\begin{document}

\title[Highly uniform prime number theorems]{Highly uniform prime number theorems}

\author{Ikuya Kaneko}
\address{Department of Mathematics, California Institute of Technology, 1200 E California~Blvd, Pasadena, CA 91125, USA}
\email{ikuyak@icloud.com}
\urladdr{\href{https://sites.google.com/view/ikuyakaneko/}{https://sites.google.com/view/ikuyakaneko/}}

\author{Jesse Thorner}
\address{Department of Mathematics, University of Illinois, Urbana, IL 61801, USA}
\email{jesse.thorner@gmail.com}

\begin{abstract}
We prove a highly uniform version of the prime number theorem for a certain class of $L$-functions.  The range of $x$ depends polynomially on the analytic conductor, and the error term is expressed in terms of an optimization problem depending explicitly on the available zero-free region.  The class contains the Rankin--Selberg $L$-function $L(s,\pi \times \pi')$ associated to cuspidal automorphic representations $\pi$ and $\pi'$ of $\mathrm{GL}_{m}$ and $\mathrm{GL}_{m'}$, respectively. Our main result implies the first uniform prime number theorems for such $L$-functions (with analytic conductor uniformity) in complete generality.
\end{abstract}

\maketitle

\section{Introduction and statement of the main result}
\label{introduction}

We prove prime number theorems for a certain class of $L$-functions possessing a Dirichlet series, Euler product, analytic continuation, and functional equation of the usual type with strong uniformity in the analytic conductor.  This problem has received attention before (see Iwaniec and Kowalski \cite[Section 5.6]{IK}), but our work provides several new and substantial improvements.  The class that we consider is slightly more restrictive than the class $\mathcal{S}(m)$ considered by Soundararajan and Thorner \cite[Sections 1.1--1.4]{SoundararajanThorner2019}.  Given an integer $m\geq 1$, $\mathcal{S}(m)$ denotes the set of $L$-functions satisfying the following four properties (A)--(D):
\begin{enumerate}[(A)]
	\item ({\it Dirichlet series and Euler product}.)  Let $p$ run over the primes.  The $L$-function $L(s,\pi)$ is given by a Dirichlet series and an Euler product
	\[
		\label{eqn:series_product}
		L(s,\pi) = \sum_{n = 1}^{\infty} \frac{\lambda_{\pi}(n)}{n^s}=\prod_{p} \prod_{j=1}^m \frac{1}{1-\alpha_{j,\pi}(p)p^{-s}},
	\]
	both converging absolutely for $\re(s)>1$.  Let $\Lambda(n)$ be the von Mangoldt function.  We define the function $a_{\pi}(n)$, supported on prime powers, by the identity
	\[
	-\frac{L'}{L}(s,\pi) = \sum_{n=1}^{\infty}\frac{a_{\pi}(n) \Lambda(n)}{n^s} = \sum_{p}\sum_{k=1}^{\infty}\frac{\sum_{j=1}^m \alpha_{j,\pi}(p)^k \log p}{p^{ks}},\qquad\re(s)>1.
	\]
	\item ({\it Analytic continuation and functional equation}.)  There exist an integer $q_{\pi}\geq 1$ attached to $\pi$, called the {\it conductor} of $\pi$, and complex numbers $\mu_{\pi}(j)$ for $1\leq j\leq m$ such that if
	\[
	r_{\pi} =  -\mathop{\mathrm{ord}}_{s=1}L(s,\pi)\in [0,m]\quad\textup{and}\quad L(s,\pi_{\infty})= \pi^{-\frac{ms}{2}}\prod_{j=1}^m \Gamma\Big(\frac{s+\mu_{\pi}(j)}{2}\Big),
	\]
	then
	\[
	\Lambda(s,\pi)= (s(1-s))^{r_{\pi}} q_{\pi}^{s/2} L(s,\pi)L(s,\pi_{\infty})
	\]
	is an entire function of order $1$.   Moreover, there exists a complex number $\kappa_{\pi}$ of modulus~$1$ such that $\Lambda(s,\pi) = \kappa_{\pi}\Lambda(1-s,\tilde{\pi})$, where
	\[
	q_{\tilde{\pi}}=q_{\pi},\qquad \{\mu_{\tilde{\pi}}(j)\}=\{\overline{\mu_{\pi}(j)}\},\qquad \{\alpha_{\tilde{\pi},j}(p)\}=\{\overline{\alpha_{\pi,j}(p)}\}.
	\]
	We define the {\it analytic conductor}
	\begin{equation}
	\label{eqn:analytic_cond}
	C(\pi)= q_{\pi}\prod_{j=1}^{m}(|\mu_{\pi}(j)|+3),
	\end{equation}
	which serves as a key measure of ``complexity'' for $L(s,\pi)$.  The zeros of $\Lambda(s,\pi)$ are the {\it nontrivial} zeros of $L(s,\pi)$, and the poles of $s^{r_{\pi}}L(s,\pi_{\infty})$ are the {\it trivial} zeros of $L(s,\pi)$.  If $p\nmid q_{\pi}$, then for all $1\leq j\leq m$, we have that $\alpha_{j,\pi}(p)\neq 0$.  If $p|q_{\pi}$, then at least one of the $\alpha_{j,\pi}(p)$ equals $0$.
	\item ({\it Pointwise bounds on local parameters}.) If $1\leq j\leq m$ and $p$ is prime, then
	\[
	|\alpha_{j,\pi}(p)|\leq p^{1-\frac{1}{m}},\qquad\re(\mu_{\pi}(j))\geq -1+\frac{1}{m}.
	\]
	\item ({\it $\ell^1$ estimates}.)  There exists a constant\footnote{The numbers $c_1, c_2, c_3, \ldots$ form a sequence of certain positive, absolute and effectively computable constants.  The notation $f\ll_{\nu} g$ or $f=O_{\nu}(g)$ means that there exists an effectively computable constant $c=c(\nu)>0$, depending at most on the parameter $\nu$, such that $|f(z)|\leq c|g(z)|$ for all $z$ in a range that is clear from context.  If no parameter $\nu$ is present, then $c$ is absolute.} $\Cl[abcon]{xrange}$ such that if $\eta>0$ and $T\geq 1$, then
	\[
	\sum_{n=1}^{\infty}\frac{|a_{\pi}(n)| \Lambda(n)}{n^{1+\eta}}\leq \frac{m}{\eta}+m\log C(\pi)+O(m^2)
	\]
	and
	\begin{equation}
	\label{eqn:new_selberg_sieve_ST}
	\sum_{x<n\leq xe^{1/T}}|a_{\pi}(n)| \Lambda(n)\ll m\frac{x}{T},\quad\textup{provided that $x\geq \Cr{xrange}m^{182 m^4} (C(\pi)T)^{144m^3}$}.
	\end{equation}
\end{enumerate}
\begin{remark}
In the version of (D) in \cite{SoundararajanThorner2019}, it is only assumed that there exist certain unspecified constants $c(m)>0$ and $c'(m)>0$, depending at most on $m$, such that
\begin{equation}
	\label{eqn:old_selberg_sieve_ST}
\sum_{x<n\leq xe^{1/T}}|a_{\pi}(n)| \Lambda(n)\leq c(m) \frac{x}{T},\quad\textup{provided that $x\geq c'(m) (C(\pi)T)^{144m^3}$}.
\end{equation}
In \eqref{eqn:new_selberg_sieve_ST}, we assume that $c(m)$ and $c'(m)$ depend on $m$ in a particular way.
\end{remark}

The $L$-functions that we consider here satisfy two additional properties that are not part of the definition of $\mathcal{S}(m)$ in \cite{SoundararajanThorner2019}.

\begin{enumerate}[(E)]
	\item ({\it Nonvanishing on $\re(s)=1$}.)  If $\re(s)\geq 1$, then $L(s,\pi)\neq 0$.  Consequently, there exists a function
	\[
	\delta_{\pi}\colon[0,\infty)\to (0,\tfrac{1}{2})
	\]
	such that if $T>0$, then $L(s,\pi)\neq 0$ in the region
	\[
	\{s\in\mathbb{C}\colon \re(s)\geq 1-\delta_{\pi}(T+3),~|\im(s)|\leq T\}
	\]
	except for at most one real zero.
\end{enumerate}
\begin{enumerate}[(F)]
	\item ({\it Zero repulsion}.)
	Let $\delta_{\pi}$ be as in (E).  Define
	\[
	\beta_0=\max(\{\beta>\max\{\tfrac{3}{4},1-\delta_{\pi}(3)\}\colon L(\beta,\pi)=0\}\cup\{\tfrac{1}{2}\}).
	\]
	If $\beta_0> \frac{1}{2}$, then
	\begin{enumerate}[(i)]
	\item $\beta_0$ is a simple zero of $L(s,\pi)$,
	\item  there exists a constant $\Cl[abcon]{Siegel_effective} \geq 1$ such that $\beta_0\leq 1-C(\pi)^{-\Cr{Siegel_effective}m}$, and
	\item there exist constants $\Cl[abcon]{Siegel_1}$ and $\Cl[abcon]{Siegel_2}$ such that if $\rho=\beta+i\gamma\neq \beta_0$ is a nontrivial zero of $L(s,\pi)$, then
	\[
	\beta\leq 1-\Cr{Siegel_1}\dfrac{\log\Big(\dfrac{\Cr{Siegel_2}}{(1-\beta_0)m\log(C(\pi)(|\gamma|+3)^m)}\Big)}{m\log(C(\pi)(|\gamma|+3)^m)}.
	\]
	\end{enumerate}
\end{enumerate}
\begin{remark}
If (E) holds, then there are infinitely many choices of $\delta_{\pi}$ such that $\beta_0=\frac{1}{2}$.  Property (F) asserts that there exists a $\delta_{\pi}$ in (E) such if $\beta_0>\frac{1}{2}$, then $\beta_0$ is a simple zero.
\end{remark}
\begin{remark}
As in \cite{SoundararajanThorner2019}, one has some latitude in the formulation of (A)--(F).  Our formulation is based on what we can prove when $L(s,\pi)$ is the $L$-function of a cuspidal automorphic representation or the Rankin--Selberg $L$-function associated to a pair of such representations.
\end{remark}

We define $\mathfrak{S}(m)$ to be the set of $L$-functions $L(s,\pi)$ that satisfy (A)--(F).  Condition (E) is equivalent to the prime number theorem for $L(s,\pi)\in\mathfrak{S}(m)$, namely
\[
\lim_{x\to\infty}\frac{1}{x}\sum_{n\leq x}a_{\pi}(n)\Lambda(n) = r_{\pi}.
\]
We prove a highly uniform version of the prime number theorem for all $L(s,\pi)\in\mathfrak{S}(m)$.
\begin{theorem}
	\label{thm:main}
	There exist constants $\Cl[abcon]{main1}\geq 1$, $\Cl[abcon]{main2}$, and $\Cl[abcon]{main3}\geq 1$ such that the following is true.  Let $m\geq 1$, and let $L(s,\pi)\in \mathfrak{S}(m)$.  Let $\delta_{\pi}(t)$ be given by (E) and $\beta_0$ by (F), and define
	\begin{equation}
	\label{eqn:eta_pi_def}
	\eta_{\pi}(x)= \inf_{t\geq 3}(\delta_{\pi}(t)\log x+ \log t).
	\end{equation}
	If $A\geq 2$ and $x\geq C(\pi)^{\Cr{main1}A^2 m^5}$, then
	\[
	\sum_{n\leq x}a_{\pi}(n) \Lambda(n) = r_{\pi}x-\frac{x^{\beta_0}}{\beta_0}+O\Big(\Big(x-\frac{x^{\beta_0}}{\beta_0}\Big)(m^5 x^{-\Cr{main2}/m^{4}}+m^{\Cr{main3}m^3}A^2 e^{-(1-\frac{1}{A})\eta_{\pi}(x)})\Big).
	\]
\end{theorem}

It is natural to compare \cref{thm:main} with the following result of Iwaniec and Kowalski, which we present in our notation using  properties (A)--(F).
\begin{theorem}[{\cite[Theorem 5.13]{IK}}]
\label{thm:IK}
There exists a constant $\Cl[abcon]{IK}$ such that the following is true. Let $L(s,\pi)$ satisfy (A), (B), (E) with
\[
\delta_{\pi}(T) = \frac{\Cr{IK}}{m^4\log(C(\pi)T)},
\]
and the $\ell^2$ estimate
\begin{equation}
\label{eqn:IK}
\sum_{n\leq x}|a_{\pi}(n)|^2 \Lambda(n)^2\ll m^2 x (\log(C(\pi)x))^2,\qquad x\geq 1.
\end{equation}
Let $\beta_0$ be as in (F).  If $x\geq 3$, then
\[
\sum_{n\leq x}a_{\pi}(n) \Lambda(n) = r_{\pi}x-\frac{x^{\beta_0}}{\beta_0} + O\Big(m^4 (\log xC(\pi))^4 x\exp\Big(-\Cr{IK}\frac{\log x}{m^4(\log C(\pi)+\sqrt{\log x})}\Big)\Big).
\]
The $O$-term is nontrivial when $x\geq C(\pi)^{4\Cr{IK}^{-1}m^4\log(m\log C(\pi))}$.
\end{theorem}
\cref{thm:main} has many advantages over \cref{thm:IK}.  First, if one thinks of $m$ as fixed (as is typical in many applications, but not all), then the range of $x$ in \cref{thm:main} depends polynomially on $C(\pi)$, unlike \cref{thm:IK}.  This is comparable with Linnik's theorem \cite{Linnik}, which states that if $q\geq 1$ and $\gcd(a,q)=1$, then there exists a constant $\Cl[abcon]{Linnik}$ such that the counting function $\pi(x;q,a)$ for primes $p\equiv a\pmod{q}$ is positive once $x\geq q^{\Cr{Linnik}}$.  Second, if $\beta_0$ is especially close to $s=1$, then the error term in~\cref{thm:main} improves, unlike \cref{thm:IK}.  This is a general extension of the zero repulsion phenomenon of Deuring and Heilbronn for Dirichlet $L$-functions, which served a crucial role in Linnik's work \cite{Linnik}.  Until now, such a quantitative manifestation of this phenomenon has only been available when $m=1$ (see \cite[Theorem 1.4]{ThornerZaman2019}).  Third, there are many important $L$-functions that are not yet known to satisfy the $\ell^2$ bound \eqref{eqn:IK} in \cref{thm:IK}, but the $\ell^1$ bounds in (D) and the pointwise bounds in (C) are known quite generally.  Fourth, \cref{thm:main} produces prime number theorems for $L$-functions having zero-free regions that are weaker than what \cref{thm:IK} assumes.

Ultimately, \cref{thm:main} reduces the problem of establishing a prime number theorem for $L(s,\pi) \in \mathfrak{S}(m)$ to the estimation of $\eta_{\pi}(x)$.  This is a straightforward optimization~calculation depending only on the available zero-free region.  This feature, as well as the improved~range of $x$, stems from our utilization of a log-free zero density estimate that follows from properties (A)--(D).  In \cref{sec:Applications}, we catalogue the most uniform versions to date of the prime number theorems that follow from \cref{thm:main} for the standard $L$-function $L(s,\pi)$ and the Rankin--Selberg $L$-function $L(s,\pi\times\pi')$ associated to cuspidal automorphic representations $\pi$ of $\mathrm{GL}_m(\mathbb{A}_{\Q})$ and $\pi'$ of $\mathrm{GL}_{m'}(\mathbb{A}_{\Q})$.  When neither $\pi$ nor $\pi'$ is self-dual, our prime number theorem for $L(s,\pi\times\pi')$ is completely new.  \cref{sec:zeros} assembles various results on zeros of $L$-functions in $\mathfrak{S}(m)$, including a log-free zero density estimate that improves as $\beta_0$ worsens.  In \cref{sec:proof_of_main_theorem}, we prove \cref{thm:main}.  The results in \cref{sec:Applications} are proved in \cref{sec:ZFR,sec:final_proofs}.

\subsection*{Acknowledgements}\label{acknowledgements}
The authors acknowledge the support of the National Science Foundation (DMS 2002265 and DMS 205118), the National Security Agency (H98230-21-1-0059), the Thomas Jefferson Fund at the University of Virginia, and the Templeton World Charity Foundation. IK thanks the Masason Foundation and the Spirit of Ramanujan STEM Talent Initiative. This research was conducted as part of the Research Experience for Undergraduates at the University of Virginia in 2021.  We thank the anonymous referee for helpful comments.

\section{Applications}
\label{sec:Applications}

Let $\mathfrak{F}_m$ denote the family of cuspidal automorphic representations of $\mathrm{GL}_m(\mathbb{A}_{\mathbb{Q}})$ possessing unitary central character, normalised so that the central character is trivial on the diagonally embedded copy of the positive reals.  Let $\pi=\pi_{\infty}\otimes(\otimes_p \pi_p)\in \mathfrak{F}_m$ have arithmetic conductor $q_{\pi}\geq 1$, where $\pi_p$ (resp. $\pi_{\infty}$) is a smooth admissible representation of $\GL_m(\Q_p)$ for every prime $p$ (resp. $\GL_m(\R)$).  The standard $L$-function $L(s,\pi)$ associated to $\pi$ can be expressed as a Dirichlet series and an Euler product, each absolutely convergent for $\re(s)>1$:
\[
L(s,\pi) = \prod_p L(s,\pi_p) = \sum_{n=1}^{\infty}\frac{\lambda_{\pi}(n)}{n^s},\qquad L(s,\pi_p)=\prod_{j=1}^m \frac{1}{1-\alpha_{j,\pi}(p)p^{-s}}.
\]
Here $\lambda_{\pi}(n)$ is the $n$-th Hecke eigenvalue of $\pi$.  When $p\nmid q_{\pi}$, the Satake isomorphism assigns to $\pi_p$ the eigenvalues $\{\alpha_{1,\pi}(p),\ldots,\alpha_{m,\pi}(p)\}$ of a certain semisimple conjugacy class in $\GL_m(\mathbb{C})$.  If $p|q_{\pi}$, then some of the $\alpha_{j,\pi}(p)$ might equal zero.  We define the numbers $a_{\pi}(n)$ by
\[
\sum_{n=1}^{\infty}\frac{a_{\pi}(n) \Lambda(n)}{n^s} = -\frac{L'}{L}(s,\pi) = \sum_{p}\sum_{k=1}^{\infty}\frac{\sum_{j=1}^m\alpha_{j,\pi}(p)^k\log p}{p^{ks}},\qquad\re(s)>1,
\]
where $\Lambda(n)$ denotes the usual von Mangoldt function.  We define $a_{\pi}(n)=0$ when $n$ is not a prime power.  If $p$ is prime, then $a_{\pi}(p)=\lambda_{\pi}(p)$.  There are $m$ Langlands parameters $\mu_{\pi}(j)$, $1\leq j\leq m$, from which we define
\[
L(s,\pi_{\infty}) = \pi^{-\frac{ms}{2}}\prod_{j=1}^m \Gamma\Big(\frac{s+\mu_{\pi}(j)}{2}\Big).
\]
If $\tilde{\pi}\in\mathfrak{F}_m$ is the contragredient representation, then $\tilde{\pi}\in\mathfrak{F}_{m}$ and
\[
q_{\tilde{\pi}}=q_{\pi},\qquad \{\mu_{\tilde{\pi}}(j)\}=\{\overline{\mu_{\pi}(j)}\},\qquad \{\alpha_{\tilde{\pi},j}(p)\}=\{\overline{\alpha_{\pi,j}(p)}\}.
\]
We denote by $\mathbbm{1}\in\mathfrak{F}_1$ the trivial representation, whose $L$-function is $\zeta(s)$.

Given $\pi\in\mathfrak{F}_m$ with conductor $q_{\pi}$ and $\pi'\in\mathfrak{F}_{m'}$ with conductor $q_{\pi'}$, consider the Rankin--Selberg $L$-function
\[
L(s,\pi\times \pi')=\prod_{p}L(s,\pi_p\times\pi_p')=\sum_{n=1}^{\infty}\frac{\lambda_{\pi\times\pi'}(n)}{n^s},
\]
absolutely convergent for $\mathrm{Re}(s)>1$, with
\[
L(s,\pi_p\times\pi_p')=\begin{cases}
	\displaystyle\prod_{j=1}^m \prod_{j'=1}^{m'} (1-\alpha_{j,\pi}(p)\alpha_{j',\pi'}(p)p^{-s})^{-1}&\mbox{if $p\nmid q_{\pi}q_{\pi'}$,}\\
	\displaystyle\prod_{j=1}^m \prod_{j'=1}^{m'} (1-\alpha_{j,j',\pi\times\pi'}(p)p^{-s})^{-1}&\mbox{if $p|q_{\pi}q_{\pi'}$.}
\end{cases}
\]
See \cite[Appendix]{SoundararajanThorner2019} for a complete description of the numbers $\alpha_{j,j',\pi\times\pi'}(p)$ when $p|q_{\pi}q_{\pi'}$.  The conductor $q_{\pi\times\pi'}$ divides $q_{\pi}^{m'}q_{\pi'}^{m}$ \cite{BushnellHenniart1997}. The $L$-function $L(s,\pi\times\pi')$ analytically continues to $\mathbb{C}$.  By our normalization of the central characters, $L(s,\pi\times\pi')$ is entire unless $\pi'=\tilde{\pi}$, in which case there is a pole of order $1$ at $s=1$.  There are $m'm$ Langlands parameters $\mu_{\pi\times\pi'}(j,j')$, with $1\leq j\leq m$ and $1\leq j'\leq m'$, such that
\[
L(s,\pi_{\infty}\times\pi_{\infty}') = \pi^{-\frac{m'ms}{2}}\prod_{j=1}^m \prod_{j'=1}^{m'}\Gamma\Big(\frac{s+\mu_{\pi,\pi'}(j,j')}{2}\Big).
\]
If $\pi_{\infty}$ and $\pi_{\infty}'$ are unramified, then
\[
\{\mu_{\pi,\pi'}(j,j')\}=\{\mu_{\pi}(j)+\mu_{\pi'}(j')\}.
\]
See \cite[Section 3]{MullerSpeh2004} for a complete description of the numbers $\mu_{\pi\times\pi'}(j,j')$ when at least one of $\pi_{\infty}$ and $\pi_{\infty}'$ is ramified.  We define the numbers $a_{\pi\times\pi'}(n)$ by the identity
\[
\sum_{n=1}^{\infty}\frac{a_{\pi\times\pi'}(n)\Lambda(n)}{n^s}=-\frac{L'}{L}(s,\pi\times\pi').
\]
The sum converges absolutely for $\re(s)>1$, and
\[
a_{\pi\times\pi'}(p)=\lambda_{\pi\times\pi'}(p).
\]

We require bounds for $C(\pi\times\pi')$ in terms of $C(\pi)$, $C(\pi')$, $C(\pi\times\tilde{\pi})$, and $C(\pi'\times\tilde{\pi}')$.
\begin{lemma}
	\label{lem:AC}
If $\pi\in\mathfrak{F}_m$ and $\pi'\in\mathfrak{F}_{m'}$, then
\[
C(\pi\times\tilde{\pi})^{\frac{m'}{4m}}C(\pi'\times\tilde{\pi}')^{\frac{m}{4m'}}\leq C(\pi\times\pi')\leq C(\pi)^{m'}C(\pi')^{m}.
\]
\end{lemma}
\begin{proof}
	By combining \cite[Lemma A.2]{Humphries2019-2} and \cite[Lemma 2.1]{SoundararajanThorner2019}, we see that
\[
C(\pi\times\tilde{\pi})^{(m')^2}C(\pi'\times\tilde{\pi}')^{m^2}\leq e^{O((m'm)^2)} C(\pi\times\pi')^{4m'm},\qquad C(\pi\times\pi')\leq e^{O(m'm)}C(\pi)^{m'}C(\pi')^{m}.
\]
In both of those proofs, the analytic conductor is defined to be
\[
C(\pi)= q_{\pi}\prod_{j=1}^{m}(|\mu_{\pi}(j)|+1),\qquad C(\pi\times\pi')= q_{\pi\times\pi'}\prod_{j=1}^{m}\prod_{j'=1}^{m'}(|\mu_{\pi\times\pi'}(j,j')|+1).
\]
A careful inspection of the proofs shows that we can remove the factors $e^{O((m'm)^2)}$ and $e^{O(m'm)}$ when the shift of $+1$ is increased to $+3$, as in \eqref{eqn:analytic_cond}.  Otherwise, the details are the same.
\end{proof}

The following result is \cite[Proposition 2.5]{SoundararajanThorner2019}.

\begin{proposition}
\label{prop:class0}
	If $\pi\in\mathfrak{F}_m$ and $\pi'\in\mathfrak{F}_{m'}$, then $L(s,\pi)\in\mathcal{S}(m)$ and $L(s,\pi\times\pi')\in\mathcal{S}(m'm)$.
\end{proposition}

We refine \cref{prop:class0} as follows.

\begin{proposition}
	\label{prop:class}
	If $\pi\in\mathfrak{F}_m$ and $\pi'\in\mathfrak{F}_{m'}$, then $L(s,\pi)\in\mathfrak{S}(m)$ and $L(s,\pi\times\pi')\in\mathfrak{S}(m'm)$.
\end{proposition}
\begin{proof}
	First, we confirm that $L(s,\pi\times\pi')\in\mathfrak{S}(m'm)$.  Properties (A), (B), and (C) are true because $L(s,\pi\times\pi')\in\mathcal{S}(m'm)$, as proved in \cite{SoundararajanThorner2019}.  The first estimate in (D) is proved in \cite[pp. 1241-1242]{SoundararajanThorner2019}.  The second estimate in (D) is proved by proceeding as in \cite[Section 6]{SoundararajanThorner2019}, but with certain specific choices of test functions $\Phi$ and $\Phi_1$.  If $\mathbf{1}_{(a,b)}(t)$ is the indicator function of the open interval $(a,b)$ and one chooses
\[
\Phi(t)=\exp\Big(\frac{4}{3}+\frac{1}{(t-\frac{1}{2})^2-1}\Big)\mathbf{1}_{(-\frac{1}{2},\frac{3}{2})}(t),\quad \Phi_1(t)=\exp\Big(1+\frac{1}{(2t-1)^2-1}\Big)\mathbf{1}_{(0,1)}(t)
\]
in the proof of \cite[Theorem 2.4]{SoundararajanThorner2019}, then bounds for the Mellin transforms of $\Phi$ and $\Phi_1$ that follow from \cite[Lemma 9]{BFI2} permit us to take
\[
c(m)\ll m,\qquad c'(m)=\Cr{xrange}m^{182 m^4}
\]
in \eqref{eqn:old_selberg_sieve_ST}.  Property (F) and a strong form of property (E) are given in Propositions \ref{prop:ZFR_SD}, \ref{prop:ZFR_Brumley}, and \ref{prop:zero_repulsion} below.  We conclude that $L(s,\pi\times\pi')\in\mathfrak{S}(m'm)$.  If $\pi'=\mathbbm{1}$, then $L(s,\pi)=L(s,\pi\times\pi')\in\mathfrak{S}(m)$.
\end{proof}

Once we incorporate the best known zero-free regions for $L(s,\pi)$ and $L(s,\pi\times\pi')$, we arrive at the most uniform versions of the prime number theorem for $L(s,\pi)$ and $L(s,\pi\times\pi')$ up to now.  First, we apply \cref{thm:main} to the standard $L$-function $L(s,\pi)$.

\begin{theorem}
	\label{thm:PNT1}
	Let $\pi\in\mathfrak{F}_m-\{\mathbbm{1}\}$.  Let
	\[
	\beta_1=\max(\{\beta>\tfrac{3}{4}\colon L(\beta,\pi)=0\}\cup\{\tfrac{1}{2}\}).
	\]
	There exist constants $\Cr{main1}$ and $\Cl[abcon]{errorterm_1}$ such that if $x\geq C(\pi)^{4\Cr{main1}m^{8}}$, then
	\[
	\sum_{n\leq x}a_{\pi}(n) \Lambda(n) = -\frac{x^{\beta_1}}{\beta_1}+O\Big(\Big(x-\frac{x^{\beta_{1}}}{\beta_{1}}\Big)\exp\Big(-\Cr{errorterm_1}\frac{\log x}{m\log C(\pi)+\sqrt{m\log x}}\Big)\Big).
	\]
\end{theorem}

For $L(s,\pi\times\pi')$, we consider two separate cases.

\begin{theorem}
	\label{thm:PNT2}
	Let $\pi\in\mathfrak{F}_m$ and $\pi'\in\mathfrak{F}_{m'}$.  Let
	\[
	\beta_1=\max(\{\beta>\tfrac{3}{4}\colon L(\beta,\pi\times\pi')=0\}\cup\{\tfrac{1}{2}\}).
	\]
	There exist constants $\Cr{main1}$ and $\Cr{errorterm_1}$ such that if
\begin{equation}
\label{eqn:special}
\pi'\in\{\tilde{\pi},\tilde{\pi}'\}
\end{equation}
and $x\geq (C(\pi)C(\pi'))^{4\Cr{main1}(m'm)^{8}}$, then
	\begin{multline*}
	\sum_{n\leq x}a_{\pi\times\pi'}(n) \Lambda(n)\\
	= r_{\pi\times\pi'}x-\frac{x^{\beta_1}}{\beta_1}+ O\Big(\Big(x-\frac{x^{\beta_1}}{\beta_1}\Big)\exp\Big(-\Cr{errorterm_1}\frac{\log x}{(m+m')\log(C(\pi)C(\pi'))+\sqrt{m(m+m')\log x}}\Big)\Big).
	\end{multline*}
\end{theorem}

All preceding prime number theorems for $L(s,\pi\times\pi')$ with a nontrivial error term and a range of $x$ with specified effective dependence on $C(\pi)$ and $C(\pi')$ have required an assumption of a ``standard'' zero-free region for $L(s,\pi\times\pi')$, which is known when \eqref{eqn:special} is true (\cref{Humphries-Thorner} below).  When \eqref{eqn:special} is true, Theorems \ref{thm:PNT1} and \ref{thm:PNT2} produce the strongest known error terms in ranges of $x$ that are polynomial in the associated analytic conductors.  When \eqref{eqn:special} is false, we only have Brumley's narrow zero-free region (\cref{Brumley} below).  For such $\pi$ and $\pi'$, \cref{thm:main} and \cref{Brumley} together imply the first prime number theorem with a nontrivial error term of any sort, with an effective range of $x$ in terms of $C(\pi)$ and $C(\pi')$.

\begin{theorem}
\label{thm:PNT3}
Let $\pi\in\mathfrak{F}_m$, $\pi'\in\mathfrak{F}_{m'}$.  If $\pi'\neq\tilde{\pi}$, then there exists a constant $\Cr{main1}>0$ such that if
\[
x\geq \exp(\Cr{main1}(C(\pi)C(\pi'))^{2(m+m')^2}),
\]
then
\[
\sum_{n\leq x}a_{\pi\times\pi'}(n)\Lambda(n) \ll_{m,m'}x(\log x)^{-\frac{1}{m'm}}.
\]
\end{theorem}

\begin{remark}
Note that if $p\nmid q_{\pi}q_{\pi'}$, then $a_{\pi\times\pi'}(p^k) = a_{\pi}(p^k)a_{\pi'}(p^k)$.  Therefore, since (C) holds for $L(s,\pi\times\pi')$ even when $p|q_{\pi}q_{\pi'}$, \cref{thm:PNT2,thm:PNT3} remain the same if we sum $a_{\pi}(n)a_{\pi'}(n)\Lambda(n)$ instead of $a_{\pi\times\pi'}(n)\Lambda(n)$.
\end{remark}

\section{zeros of $L$-functions in $\mathfrak{S}(m)$}
\label{sec:zeros}

Let $m\geq 1$ be an integer, and let $\pi\in\mathfrak{S}(m)$.  Since $\Lambda(s,\pi)$ is entire of order $1$ by (B), there exist constants $a_{\pi}, b_{\pi} \in \mathbb{C}$ such that we have the Hadamard factorisation
\begin{equation}\label{Hadamard-product}
\Lambda(s,\pi) = e^{a_{\pi}+b_{\pi} s} \prod_{\Lambda(\rho,\pi) = 0} \Big(1-\frac{s}{\rho} \Big) e^{\frac{s}{\rho}}.
\end{equation}

\begin{lemma}\label{lem:number-of-zeros}
If $L(s,\pi)\in\mathfrak{S}(m)$, $t \in \R$, and $0<\eta\leq 2$, then
\[
\#\{\rho\colon |\rho-(1+it)|\leq\eta,~L(\rho,\pi) = 0 \} \ll \eta m \log(C(\pi)(2+|t|))+m^2,
\]
where the zeros $\rho$ are counted with multiplicity.  In particular,
\[
\#\{\rho = \beta+i\gamma: 0 < \beta < 1, |\gamma-t| \leq 1, L(\rho,\pi) = 0 \} \ll m \log(C(\pi)(2+|t|)),
\]
\end{lemma}
\begin{proof}
	Since $\mathfrak{S}(m)\subseteq \mathcal{S}(m)$, this follows from \cite[Lemma 3.1]{SoundararajanThorner2019} when $0<\eta\leq 1$.  Otherwise, this follows from \cite[Proposition 5.7]{IK}.
\end{proof}

\noindent
Next, we refine the $m$-dependence for the log-free zero density estimate in \cite[Theorem 1.2]{SoundararajanThorner2019}.

\begin{theorem}
\label{thm:LFZDE}
	Let $L(s,\pi)\in\mathfrak{S}(m)$ and $T\geq 1$.  For $\sigma\geq 0$, define
	\[
	N_{\pi}(\sigma,T)= \#\{\rho=\beta+i\gamma \colon L(\rho,\pi)=0,~\beta\geq \sigma,~|\gamma|\leq T\},
	\]
	where each $\rho$ is counted with multiplicity.  There exists a constant $\Cl[abcon]{LFZDE}$ such that
	\[
	N_{\pi}(\sigma,T)\ll m^{\Cr{LFZDE} m^3}(C(\pi)T)^{10^7 m^3(1-\sigma)}.
	\]
\end{theorem}
\begin{proof}
The proof proceeds as in \cite[Section 4]{SoundararajanThorner2019} with three small modifications.  First, we use the bound \eqref{eqn:new_selberg_sieve_ST} instead of the bound \eqref{eqn:old_selberg_sieve_ST} (cf. \cite[(1.10)]{SoundararajanThorner2019}).  This helps us to explicate the suppressed $m$-dependence in the implied constant in the third-to-last equation on \cite[p. 1252]{SoundararajanThorner2019}.  Second, we require that $\eta$ in \cite[Proof of Theorem 1.2]{SoundararajanThorner2019} satisfy
\[
\frac{1}{200\log(C(\pi)T)}<\eta\leq \frac{1}{200m}\quad\textup{instead of}\quad\frac{1}{\log(C(\pi)T)}<\eta\leq\frac{1}{200m}.
\]
When $T=1$, this ensures that the interval containing $\eta$ is always nonempty, even if $C(\pi)<e^{200m}$.  (Since $m$ was implicitly assumed to be fixed in \cite{SoundararajanThorner2019}, such considerations were inconsequential.)  Third, one chooses
\[
K = 10^5 m^3\eta \log(C(\pi)T)+300m^3\log(em)+\Cl[abcon]{Kconst}m^2
\]
in \cite[(4.4)]{SoundararajanThorner2019}, where $\Cr{Kconst}$ is suitably large.  This ensures that the range of $x$ in \eqref{eqn:new_selberg_sieve_ST} is compatible with the range of integration in the $x$-integral two equations below \cite[(4.6)]{SoundararajanThorner2019}, even when $m$ is not fixed.  These modifications allow us to determine the dependence of the implied constant in \cite[Theorem 1.2]{SoundararajanThorner2019} on $m$.
\end{proof}

We use (F) to refine \cref{thm:LFZDE}.

\begin{corollary}
\label{cor:LFZDE_repulsion}
	Let $L(s,\pi)\in\mathfrak{S}(m)$.  For $\sigma\geq 0$ and $T\geq 1$, define
	\[
	N_{\pi}^*(\sigma,T)= \begin{cases}
\#\{\rho=\beta+i\gamma \neq \beta_0 \colon L(\rho,\pi)=0,~\beta\geq \sigma,~|\gamma|\leq T\}&\mbox{if $\beta_0>\frac{1}{2}$,}\\
N_{\pi}(\sigma,T)&\mbox{otherwise,}
\end{cases}
	\]
	where each $\rho$ is counted with multiplicity.  Let $\beta_0$ be as in \cref{thm:main}, and  define
	\[
	\nu_{\pi}(T)= \min\{1,(1-\beta_{0})\log(C(\pi)T)\}.
	\]
	There exists a constant $\Cl[abcon]{LFZDE_repulsion}\geq 1$ such that
	\[
	N_{\pi}^*(\sigma,T) \ll \nu_{\pi}(T) m^{\Cr{LFZDE_repulsion}m^3}(C(\pi)T)^{\Cr{LFZDE_repulsion} m^3(1-\sigma)}.
	\]
\end{corollary}
\begin{proof}
If $\beta_0=\frac{1}{2}$ or $(1-\beta_0)m\log(C(\pi)T^m)\geq \frac{\Cr{Siegel_2}}{e}$, then the result follows from \cref{thm:LFZDE}.  Now, suppose that
\[
\beta_0>\frac{1}{2},\qquad (1-\beta_0)m\log(C(\pi)T^m)<\frac{\Cr{Siegel_2}}{e}.
\]
If
	\[
	\sigma > 1-\Cr{Siegel_1}\dfrac{\log\Big(\dfrac{\Cr{Siegel_2}}{(1-\beta_0)m\log(C(\pi)(|\gamma|+3)^m)}\Big)}{m\log(C(\pi)(|\gamma|+3)^m)}
	\]
	then by (F), we have that $N_{\pi}^*(\sigma,T)=0$.  Otherwise, $\sigma$ satisfies
	\begin{equation}
	\label{eqn:nu_lower}
	\frac{\Cr{Siegel_2}}{m^{2}}(C(\pi)(T+3)^m)^{-\frac{m}{\Cr{Siegel_1}}(1-\sigma)}\leq \frac{1-\beta_0}{m}\log(C(\pi)(T+3)^m)\ll \nu_{\pi}(T).
	\end{equation}
	It follows from \cref{thm:LFZDE} that $N_{\pi}^{*}(\sigma,T)\leq N_{\pi}(\sigma,T)$ is
	\[
	\ll m^{\Cr{LFZDE}m^3}(C(\pi)T)^{10^7 m^3(1-\sigma)}=\nu_{\pi}(T)m^{\Cr{LFZDE}m^3}(C(\pi)T)^{10^7 m^3(1-\sigma)}\nu_{\pi}(T)^{-1}.
	\]
	Bounding $\nu_{\pi}(T)^{-1}$ using \eqref{eqn:nu_lower}, we obtain the corollary.
\end{proof}

\section{Proof of \cref{thm:main}}
\label{sec:proof_of_main_theorem}

Let $L(s,\pi)\in\mathfrak{S}(m)$.  We will prove \cref{thm:main} when $\beta_0>\frac{1}{2}$ in (E), in which case (F) states that $\beta_0$ is a real simple zero of $L(s,\pi)$.  If $\beta_0=\frac{1}{2}$, then the proof is easier.

\subsection{Preliminaries}

We use the following smooth weight function.
\begin{lemma}
\label{lem:WeightChoice}
Let $x \geq 3$, $\epsilon \in (0,\frac{1}{4})$, and an integer $\ell \geq 2$. Define $B = \epsilon/(2 \ell \log x)$. There exists a continuous function $f(t)  = f(t; x, \ell, \epsilon)$ of a real variable $t$ such that:
\begin{enumerate}[(i)]
\item $0 \leq f(t) \leq 1$ for all $t \in \R$, and $f(t) \equiv 1$ for $\tfrac{1}{2} \leq t \leq 1$.
\item The support of $f$ is contained in the interval $[\tfrac{1}{2}-\frac{\epsilon}{\log x}, 1+\frac{\epsilon}{\log x}]$.
\item Its Laplace transform $F(z) = \int_{\R} f(t) e^{-zt} dt$ is entire and given by
\[
F(z) = e^{-(1+ 2\ell B)z} \cdot \Big(\frac{1-e^{(\frac{1}{2}+2\ell B)z}}{-z} \Big) \Big(\frac{1-e^{2Bz}}{-2Bz} \Big)^{\ell}.
\]
\item Let $s = \sigma + it$, $\sigma > 0$, $t \in \R$ and $\alpha$ be any real number satisfying $0 \leq \alpha \leq \ell$. Then
\[
|F(-s \log x)| \leq \frac{e^{\sigma \epsilon} x^{\sigma}}{|s| \log x} \cdot 
(1+x^{-\sigma/2} ) \cdot \Big(\frac{2\ell}{\epsilon |s|} \Big)^{\alpha}.
\]
Moreover, $|F(-s \log x)| \leq e^{\sigma \epsilon} x^{\sigma}$ and $1/2 < F(0) < 3/4$.
\item If $\frac{3}{4}<\sigma\leq 1$ and $x\geq 10$, then
\[
F(-\log x) - F(-\sigma\log x)=\Big(\frac{x}{\log x}-\frac{x^{\sigma}}{\sigma \log x}\Big)(1+O(\epsilon))+O\Big(\frac{x^{1/2}}{\log x}\Big).
\]
\end{enumerate}
\end{lemma}
\begin{proof}
This is contained in the statement of \cite[Lemma 2.2]{ThornerZaman2019}.
\end{proof}

Using \cref{lem:WeightChoice} and (C), we closely approximate
\[
\sum_{n\leq x}a_{\pi}(n)\Lambda(n)
\]
with a smoothed sum.

\begin{lemma}
\label{lem:unsmoothing}
Let $\pi\in\mathfrak{S}(m)$ and
\[
x \geq \Cr{xrange}^{145}m^{26390m^4} C(\pi)^{20880m^3},\qquad 0<\epsilon<\min\{x^{-\frac{1}{145m^3}},\tfrac{1}{4}\}.
\]
If $f$ is given by \cref{lem:WeightChoice}, then
\[
\Big|\sum_{n \leq x} a_{\pi}(n) \Lambda(n)-\sum_{n = 1}^{\infty} a_{\pi}(n) \Lambda(n) f \Big(\frac{\log n}{\log x} \Big) \Big|\ll mx^{1-\frac{1}{2m}}+\varepsilon x.
\]
\end{lemma}

\begin{proof}
By hypothesis, we have $0<\epsilon<\frac{1}{4}$.  As such, \cref{lem:WeightChoice} renders the equality
\[
\sum_{n \leq x} a_{\pi}(n) \Lambda(n)=\sum_{n = 1}^{\infty} a_{\pi}(n) \Lambda(n) f \Big(\frac{\log n}{\log x} \Big)
+O\Big(\sum_{\substack{1\leq n \leq \sqrt{x} \\ x\leq n\leq xe^{\epsilon}}}|a_{\pi}(n)| \Lambda(n)\Big).
\]
We apply (A), (C), and (D) with $T=\epsilon^{-1}$, the prime number theorem $\sum_{n\leq x}\Lambda(n)\sim x$, and partial summation to obtain
\[
\Big(\sum_{n \leq \sqrt{x}}+\sum_{x<n\leq xe^{\epsilon}}\Big) |a_{\pi}(n)| \Lambda(n) \ll m\sum_{n\leq \sqrt{x}}n^{1-\frac{1}{m}}\Lambda(n)+\epsilon m x\ll mx^{1-\frac{1}{2m}}+\epsilon m x.\qedhere
\]
\end{proof}

We proceed to asymptotically evaluate the smoothed sum of $a_{\pi}(n) \Lambda(n)$.  We let $\rho=\beta+i\gamma$ run through the nontrivial zeros of $L(s,\pi)$, and $\sum_{\rho}'$ denotes a sum over $\rho\neq\beta_0$, where each zero is counted with multiplicity.
\begin{lemma}
\label{lem:residue}
If $x \geq 3$ and $\ell\geq m^3$, then
\begin{multline*}
\frac{1}{\log x} \sum_{n = 1}^{\infty} a_{\pi}(n) \Lambda(n) f \Big(\frac{\log n}{\log x} \Big)=r_{\pi}F(-\log x)-F(-\beta_0 \log x)\\
-\sideset{}{'}\sum_{|\rho|>\frac{1}{4}} F(-\rho \log x)+ O \Big(\Big(\frac{\ell}{\epsilon}\frac{x^{1-\frac{1}{2m}}}{\log x}+mx^{\frac{1}{4}}\Big)\log C(\pi)\Big).
\end{multline*}
\end{lemma}

\begin{proof}
By Laplace inversion and (B), we obtain the identity
\begin{equation}
\label{integral}
\begin{aligned}
\frac{1}{\log x}&\sum_{n = 1}^{\infty} a_{\pi}(n) \Lambda(n) f \Big(\frac{\log n}{\log x} \Big)\\
&= \frac{1}{2\pi i} \int_{3-i\infty}^{3+i\infty} -\frac{L^{\prime}}{L}(s,\pi) F(-s \log x) ds\\
&=\frac{1}{2\pi i} \int_{3-i\infty}^{3+i\infty} \Big(\frac{r_{\pi}}{s-1}+\frac{r_{\pi}}{s}+\frac{\log q_{\pi}}{2}+\frac{L'}{L}(s,\pi_{\infty})-\frac{\Lambda^{\prime}}{\Lambda}(s,\pi)\Big) F(-s \log x) ds.
\end{aligned}
\end{equation}
By \cref{lem:WeightChoice}, $F$ is entire and decays rapidly in vertical strips. By (C), we have that $-\frac{L'}{L}(s,\pi_{\infty})$ is holomorphic for $\Re(s)> 1-\frac{1}{m}$.  It follows that \eqref{integral} equals
\begin{multline*}
r_{\pi}F(-\log x)-\frac{1}{2\pi i}\int_{3-i\infty}^{3+i\infty}\frac{\Lambda'}{\Lambda}(s,\pi)F(-s\log x)ds\\
+\frac{1}{2\pi i}\int_{1-\frac{1}{2m}-i\infty}^{1-\frac{1}{2m}+i\infty}\Big(\frac{r_{\pi}}{s-1}+\frac{r_{\pi}}{s}+\frac{\log q_{\pi}}{2}+\frac{L'}{L}(s,\pi_{\infty})\Big)F(-s\log x)ds.
\end{multline*}

By (A), we have that $r_{\pi}\in[0,m]$.  Using Stirling's formula and (C), it follows that
\[
\Big|\frac{r_{\pi}}{s-1}+\frac{r_{\pi}}{s}+\frac{\log q_{\pi}}{2}+\frac{L'}{L}(s,\pi_{\infty})\Big|\ll m^2+m\log(|\im(s)|+3)+\log C(\pi),\qquad \Re(s)=1-\frac{1}{2m}.
\]
Therefore, by an application of \cref{lem:WeightChoice}(iv) (with $\alpha=0$ when $|\im(s)|\leq m$ and $\alpha=1$ when $|\im(s)|>m$), we observe that
\begin{align*}
&\Big|\frac{1}{2\pi i}\int_{1-\frac{1}{2m}-i\infty}^{1-\frac{1}{2m}+i\infty}\Big(\frac{r_{\pi}}{s-1}+\frac{r_{\pi}}{s}+\frac{\log q_{\pi}}{2}+\frac{L'}{L}(s,\pi_{\infty})\Big)F(-s\log x)ds\Big|\\
&\ll \frac{x^{1-\frac{1}{2m}}}{\log x}\int_{-m}^{m}(m^2+m\log(|t|+3)+\log C(\pi))dt\\
&+\frac{\ell x^{1-\frac{1}{2m}}}{\epsilon\log x}\int_{|t|>m}(m^2+m\log(|t|+3)+\log C(\pi))\frac{dt}{|t|^2}\\
&\ll \frac{x^{1-\frac{1}{2m}}}{\log x}(m^3+m\log C(\pi))+\frac{\ell}{\epsilon m}\frac{x^{1-\frac{1}{2m}}}{\log x}(m^2+\log C(\pi))\\
&\ll \frac{\ell}{\epsilon}\frac{x^{1-\frac{1}{2m}}}{\log x}\log C(\pi).
\end{align*}
Consequently, by the residue theorem, \eqref{integral} equals
\[
r_{\pi}F(-\log x)-F(-\beta_0 \log x)-\sideset{}{'}\sum_{\rho}F(-\rho\log x)+O\Big(\frac{\ell}{\epsilon}\frac{x^{1-\frac{1}{2m}}}{\log x}\log C(\pi)\Big).
\]
For the zeros $\rho$ such that $|\rho|\leq\frac{1}{4}$, Lemmata \ref{lem:number-of-zeros} and \ref{lem:WeightChoice}(iv) imply that
\[
\sum_{\substack{ |\rho| \leq \frac{1}{4}}} |F(-\rho \log x)| \ll x^{\frac{1}{4}} \#\{\rho\colon |\rho|<\tfrac{1}{4}\}\ll mx^{\frac{1}{4}}\log C(\pi).
\]
The lemma follows once we combine the estimates above.
\end{proof}

\subsection{Estimating the sum over zeros}
\label{estimating-the-zeros}

We are in a position to evaluate the sum over nontrivial zeros $\rho$ in Lemma \ref{lem:residue} using the log-free zero density estimate in \cref{cor:LFZDE_repulsion}.
\begin{lemma}
\label{remaining}
Let
\begin{equation}
\label{eqn:param_ranges_1}
A\geq 2,\qquad \ell = A\Cr{LFZDE_repulsion}m^3,\qquad \varepsilon = \min\{\tfrac{1}{5},2A\ell x^{-1/(2A\ell)}\}.	
\end{equation}
Let $\delta_{\pi}$ be as in (E), and let $\eta_{\pi}(x)$ be as in \eqref{eqn:eta_pi_def}.  Let $\nu_{\pi}(T)$ be as in \cref{cor:LFZDE_repulsion}.  If
\begin{equation}
\label{eqn:param_ranges_2}
x\geq C(\pi)^{2A^2\Cr{LFZDE_repulsion}m^3},
\end{equation}
then
\[
\sideset{}{'}\sum_{|\rho| \geq \frac{1}{4} } |F(-\rho \log x)| 
\ll A^2\nu_{\pi}(1)m^{\Cr{LFZDE_repulsion}m^3} \frac{x}{\log x} e^{-(1-\frac{1}{A})\eta_{\pi}(x)}.
\]
\end{lemma}

\begin{proof}
Let $T_0=0$, and for $j \geq 1$, let $T_{j} = 2^{j-1}$.  Consider the sum
\begin{equation}\label{Z}
Z_{j} = \frac{\log x}{x} \sideset{}{'}\sum_{\substack{|\rho|\geq\frac{1}{4} \\ T_{j-1} \leq |\gamma| \leq T_{j} }} |F(-\rho \log x)|.
\end{equation}
First, we estimate the contribution of each zero $\rho$ appearing in $Z_{j}$. Let $\rho = \beta+i\gamma$ satisfy $T_{j-1} \leq |\gamma| \leq T_{j}$ and $|\rho| \geq \frac{1}{4}$, so that $|\rho| \geq \max(T_{j-1}, 1/4) \geq T_{j}/4$ and $|\rho| \geq \frac{1}{13}(|\gamma|+3)$. Therefore, by \cref{lem:WeightChoice}(iv) with $\alpha = \ell(1-\beta)$ and our choice of $\varepsilon$, we have that
\[
\frac{\log x}{x} |F(-\rho \log x)| \ll \frac{x^{\beta-1}}{|\rho|} \Big(\frac{2\ell}{\varepsilon|\rho|} \Big)^{\ell(1-\beta)} \ll T_{j}^{-\frac{1}{A}}(|\gamma|+3)^{-(1-\frac{1}{A})}x^{-(1-\beta)(1-\frac{1}{A})}(x^{\frac{1}{2A}} T_{j}^{\ell})^{-(1-\beta)}.
\]
By \eqref{eqn:param_ranges_1} and \eqref{eqn:param_ranges_2}, we have that
\begin{equation}\label{specifications}
\frac{\log x}{x} |F(-\rho \log x)| \ll T_{j}^{-\frac{1}{A}}(|\gamma|+3)^{\frac{1}{A}-1} x^{-(1-\beta)(1-\frac{1}{A})}(C(\pi) T_{j})^{-A\Cr{LFZDE_repulsion}(1-\beta)m^3}.
\end{equation}
From (E) and \eqref{eqn:eta_pi_def}, one has
\begin{equation}
\label{eqn:etabound}
(|\gamma|+3)^{\frac{1}{A}-1} x^{-(1-\beta)(1-\frac{1}{A})} 
 \leq e^{-(1-\frac{1}{A})\eta_{\pi}(x)}.
\end{equation}
Combining  \eqref{Z}, \eqref{specifications}, and \eqref{eqn:etabound}, we derive
\[
Z_{j} \ll e^{-(1-\frac{1}{A})\eta_{\pi}(x)} T_{j}^{-\frac{1}{A}} \sideset{}{'}\sum_{T_{j-1} \leq |\gamma| \leq T_{j}} 
(C(\pi) T_{j})^{-A\Cr{LFZDE_repulsion}(1-\beta)m^3}.
\]
By partial summation and \cref{cor:LFZDE_repulsion}, it follows that
\[
\sideset{}{'}\sum_{T_{j-1} \leq |\gamma| \leq T_{j}} 
(C(\pi) T_{j})^{-A\Cr{LFZDE_repulsion}(1-\beta)m^3} \ll \int_{0}^{1} (C(\pi) T_{j})^{-A\Cr{LFZDE_repulsion} m^3\alpha} dN_{\pi}^{\ast}(1-\alpha, T_{j})
\ll m^{\Cr{LFZDE_repulsion}m^3}\nu_{\pi}(T_{j}).
\]
Observe that
\[
\nu_{\pi}(T_{j}) T_{j}^{-\frac{1}{2A}} \leq (1-\beta_0 ) \sup_{t \geq 1} \{t^{-\frac{1}{2A}} \log(C(\pi)t) \} \ll A\nu_{\pi}(1).
\]
The lemma now follows from the bound
\[
\sum_{j = 1}^{\infty} Z_{j} \ll A\nu_{\pi}(1)m^{\Cr{LFZDE_repulsion}m^3} e^{-(1-\frac{1}{A})\eta_{\pi}(x)} \sum_{j = 0}^{\infty} 2^{-\frac{j-1}{2A}} 
\ll A^2\nu_{\pi}(1)m^{\Cr{LFZDE_repulsion}m^3} e^{-(1-\frac{1}{A})\eta_{\pi}(x)}.\qedhere
\]
\end{proof}

\begin{lemma}
\label{lem:nu_pi}
	If $x \geq C(\pi)^{1056\Cr{Siegel_effective}\Cr{LFZDE_repulsion}m^5}$, then
	\[
	x^{1-\frac{1}{1056\Cr{LFZDE_repulsion}m^4}}\ll \nu_{\pi}(1)x\ll x-\frac{x^{\beta_0}}{\beta_0}.
	\]
\end{lemma}
\begin{proof}
	It suffices to prove the lemma when $\nu_{\pi}(1)=(1-\beta_0)\log C(\pi)<1$.  We consider two cases.  First, if $(1-\beta_0)\log x\geq 1$, then
	\[
	\nu_{\pi}(1)x \ll x\ll x(1-2e^{-1})\leq x\Big(1-\frac{x^{-(1-\beta_0)}}{\beta_0}\Big)=x-\frac{x^{\beta_0}}{\beta_0}.
	\]
	
	Second, assume that $0<(1-\beta_0)\log x<1$.  Our hypothesis on the range of $x$ implies that $x\geq e^4$.  We claim that
	\begin{equation}
	\label{eqn:referee}
	\frac{(1-\beta_0)\log(x/e)}{1-e^{-(1-\beta_0)\log x}/\beta_0}\leq \frac{e}{e-1},	
	\end{equation}
	from which we deduce the desired bound
	\[
	\nu_{\pi}(1)x= (1-\beta_0)x\log C(\pi)\ll (1-\beta_0)x\log\frac{x}{e}\ll x\Big(1-\frac{e^{-(1-\beta_0)\log x}}{\beta_0}\Big)=x-\frac{x^{\beta_0}}{\beta_0}.
	\]
	To finish the proof of the lemma, we observe that $C(\pi)^{-\Cr{Siegel_effective}m}\ll \nu_{\pi}(1)$ by (F).  Now, the lemma now follows from our range of $x$.
	
	To prove the claimed bound in \eqref{eqn:referee}, we make the change of variables $(1-\beta_0)\log x = t$, in which case the left hand side of \eqref{eqn:referee} equals
	\[
	f(x,t)=\frac{e^t t(\log x-t)(\log x-1)}{(e^t (\log x-t)-\log x)\log x}.
	\]
	We maximize $f(x,t)$ when $x\geq e^4$ and $0<t\leq 1$.  Observe that
	\[
	\lim_{t\to 0^+}f(x,t)=1\leq \frac{e(\log x-1)^2}{((e-1)\log x-e)\log x}=\lim_{t\to 1^-}f(x,t),
	\]
	and the sign of $\frac{d}{dt}f(x,t)$ for $t\in(0,1]$ is the same as the sign of
	\begin{multline*}
	(e^t-(t+1))(\log x)^2-t(2e^t-(t+2))\log x+e^t t^2\\
	\geq ((e^t-(t+1))\log x-t(2e^t-(t+2)))\log x\\
	\geq (4(e^t-(t+1))-t(2e^t-(t+2)))\log x\geq 0.	
	\end{multline*}
	Thus, as $t$ monotonically increases from 0 to 1, $f(x,t)$ monotonically increases from 1 to
	\[
	\frac{e(\log x-1)^2}{((e-1)\log x-e)\log x}.
	\]
	It follows that
	\[
	\sup_{\substack{t\in(0,1],~x\geq e^4}}f(x,t) = \sup_{x\geq e^4}\frac{e(\log x-1)^2}{((e-1)\log x-e)\log x}=\lim_{x\to\infty}\frac{e(\log x-1)^2}{((e-1)\log x-e)\log x}=\frac{e}{e-1}.\qedhere
	\]
\end{proof}

\subsection{Proof of \cref{thm:main}}
Without loss of generality, we may assume that $\frac{3}{4}<\beta_0<1$.  Let $\Cr{main1}$ be suitably large, and let $A\geq 4$.  If $x\geq C(\pi)^{\Cr{main1}A^2 m^5}$, then by Lemmata \ref{lem:unsmoothing}--\ref{remaining},
\begin{multline*}
\sum_{n\leq x}a_{\pi}(n)\Lambda(n)=(r_{\pi}F(-\log x)-F(-\beta_0\log x))\log x\\
+O\Big(\nu_{\pi}(1)x\Big(\frac{m}{\nu_{\pi}(1)x^{\frac{1}{2m}}}+\frac{\epsilon}{\nu_{\pi}(1)}+\frac{\ell\log C(\pi)}{\epsilon\nu_{\pi}(1)x^{\frac{1}{2m}}}+m^{\Cr{LFZDE_repulsion}m^3}A^2 e^{-(1-\frac{1}{A})\eta_{\pi}(x)}\Big)\Big).
\end{multline*}
By \cref{lem:nu_pi} and the choices of $\ell$ and $\epsilon$ in \cref{remaining}, the $O$-term is
\[
\ll\nu_{\pi}(1)x(m^4 x^{-\frac{1}{33\Cr{LFZDE_repulsion}m^4}}+m^{\Cr{LFZDE_repulsion}m^3}A^2 e^{-(1-\frac{1}{A})\eta_{\pi}(x)}).
\]
By \cref{lem:WeightChoice}(iii), if $\frac{3}{4}<\sigma\leq 1$, then
\[
F(-\sigma\log x)\log x=\frac{x^{\sigma}}{\sigma}\Big(\frac{e^{\epsilon\sigma/\ell}-1}{\epsilon\sigma/\ell}\Big)^{\ell}+O(x^{\frac{\sigma}{2}})=\frac{x^{\sigma}}{\sigma}(1+O(\epsilon \sigma))+O(x^{\frac{\sigma}{2}}).
\]
This bound, along with \cref{lem:WeightChoice}(v), implies that
\[
(r_{\pi}F(-\log x)-F(-\beta_0\log x))\log x = r_{\pi}x-\frac{x^{\beta_0}}{\beta_0}+O(m(\epsilon x+\sqrt{x})).
\]
Our choice of $\epsilon$ and the lower bound for $\nu_{\pi}(1)x$ in \cref{lem:nu_pi} imply that
\[
r_{\pi}(\epsilon x+\sqrt{x})\ll m \epsilon x\ll m^5\nu_{\pi}(1)x^{1-\frac{1}{33\Cr{LFZDE_repulsion}m^4}},
\]
from which we conclude that
\[
\sum_{n\leq x}a_{\pi}(n)\Lambda(n) = r_{\pi}x-\frac{x^{\beta_0}}{\beta_0}+O(\nu_{\pi}(1)x(m^5 x^{-\frac{1}{33\Cr{LFZDE_repulsion}m^4}}+m^{\Cr{LFZDE_repulsion}m^3}A^2 e^{-(1-\frac{1}{A})\eta_{\pi}(x)})).
\]
To finish the proof, we invoke the upper bound for $\nu_{\pi}(1)x$ in \cref{lem:nu_pi}.

\section{Properties (E) and (F) for Rankin--Selberg $L$-functions}
\label{sec:ZFR}

Let $\pi\in\mathfrak{F}_m$ and $\pi'\in\mathfrak{F}_{m'}$.  We now compile the best known zero-free regions for $L(s,\pi\times\pi')$.

\begin{proposition}
	\label{prop:ZFR_SD}
	There exists a constant $\Cl[abcon]{ZFR}$ such that if $\pi\in\mathfrak{F}_m$ and $\pi'\in\mathfrak{F}_{m'}$ satisfy \eqref{eqn:special}, then $L(s,\pi\times\pi')\neq 0$ in the region
	\[
	\re(s)\geq 1-\frac{\Cr{ZFR}}{(m+m')\log(C(\pi)C(\pi')(|\im(s)|+3)^{m})}
	\]
	apart from at most one exceptional zero $\beta_1<1$.  If $\beta_1$ exists, then $\beta_1$ is both real and simple, and
	\[
	\text{$\pi=\tilde{\pi}~$ and $~\pi'=\tilde{\pi}'$,}\qquad\text{or}\qquad\pi'=\tilde{\pi}.
	\]
\end{proposition}
\begin{remark}
This implies a zero-free region for $L(s,\pi)=L(s,\pi\times\mathbbm{1})$.  If $\beta_1$ exists, then $\pi=\tilde{\pi}$.
\end{remark}

\begin{proof}
	When $\pi'=\tilde{\pi}$, this is \cite[Theorem 2.1(1)]{HumphriesThorner2020}.  When $\pi'=\tilde{\pi}'$, this is \cite[Theorem A.1]{Humphries2019-2} with a small improvement in the dependence on $m$ and $m'$ stemming from the fact that if $\Pi$ is the isobaric automorphic representation $\pi\otimes|\det|^{i\gamma}\boxplus\tilde{\pi}\otimes|\det|^{-i\gamma}\boxplus\pi'$, then the Dirichlet coefficients of $\log L(s,\Pi\times\tilde{\Pi})$ are nonnegative \cite[Lemma a]{Hoffstein}.  This produces an improved degree dependence in \cite[Lemma 5.9]{IK} that we insert into the proof of \cite[Theorem A.1]{Humphries2019-2}.
\end{proof}

\begin{proposition}
	\label{prop:ZFR_Brumley}
	Let $\pi\in\mathfrak{F}_m$ and $\pi'\in\mathfrak{F}_{m'}$.  Assume that $\pi'\neq\tilde{\pi}$.  For all $\epsilon>0$, there exists an effectively computable constant $c_{m,m',\epsilon}>0$ such that $L(s,\pi\times\pi')\neq 0$ in the region
	\begin{equation}
	\label{eqn:ZFRBrumley}
	\re(s)\geq 1-\frac{c_{m,m',\epsilon}}{((C(\pi)C(\pi'))^{m+m'}(3+|t|)^{m'm})^{1-\frac{1}{m+m'}+\frac{\epsilon}{2}}}.
	\end{equation}
\end{proposition}
\begin{proof}
This follows from \cite[Theorem A.1]{Lapid} and \cref{lem:AC}.
\end{proof}

Finally, property (F) for $L(s,\pi)$ and $L(s,\pi\times\pi')$ follows from the next result.
\begin{proposition}
\label{prop:zero_repulsion}
Let $\pi \in \mathfrak{F}_{m}$ and $\pi' \in \mathfrak{F}_{m'}$. If $\beta_0>\frac{1}{2}$ is a real simple zero of $L(s,\pi\times\pi')$, then $\beta_0 \leq 1- C(\pi\times\pi')^{-\Cr{Siegel_effective}m'm}$, and apart from $s = \beta_0 $, $L(s, \pi \times \pi')$ is nonzero in the region
\[
\Re(s) \geq 1-\Cr{Siegel_1} \dfrac{\log \Big(\dfrac{\Cr{Siegel_2}}{(1-\beta_0 )
m'm\log(C(\pi\times\pi')(|\Im(s)|+3)^{m'm})}\Big)}
{m'm\log(C(\pi\times\pi')(|\Im(s)|+3)^{m'm})}.
\]
\end{proposition}

\begin{proof}
When $\pi'=\tilde{\pi}$, this was shown in \cite[Proposition 5.3 and Corollary 5.4]{HumphriesThorner2020}.  When $\pi'\neq\tilde{\pi}$, one applies the same ideas in \cite[Proposition 5.3 and Corollary 5.4]{HumphriesThorner2020} to the $L$-function
\[
D(s) = L(s,\pi\times\tilde{\pi})L(s,\pi'\times\tilde{\pi}')L(s,\pi\times\pi')L(s,\tilde{\pi}\times\tilde{\pi}')
\]
instead of $L(s,\pi\times\tilde{\pi})$, which has nonnegative Dirichlet coefficients by \cite[Lemma a]{Hoffstein}.  The key observation is that while $D(s)$ has a pole of order 2 at $s=1$, if $\rho$ is a nontrivial zero of $L(s,\pi\times\pi')$, then $\bar{\rho}$ is a nontrivial zero of $L(s,\tilde{\pi}\times\tilde{\pi}')$.  It remains to bound the analytic conductor of $D(s)$ in terms of $C(\pi\times\pi')$, which is accomplished using \cref{lem:AC}.
\end{proof}

\section{Proofs of prime number theorems for $L(s,\pi)$ and $L(s,\pi\times\pi')$}
\label{sec:final_proofs}

Let $\pi\in\mathfrak{F}_m$ and $\pi'\in\mathfrak{F}_{m'}$.  To prove Theorems \ref{thm:PNT1}, \ref{thm:PNT2}, and \ref{thm:PNT3}, it remains (after invoking \cref{thm:main}) to bound $e^{-\eta_{\pi}(x)}$ and $e^{-\eta_{\pi\times\pi'}(x)}$ for $x\geq 3$ using \cref{prop:ZFR_SD,prop:ZFR_Brumley}.

\begin{lemma}\label{Humphries-Thorner}
If $\pi\in\mathfrak{F}_m$ and $\pi'\in\mathfrak{F}_{m'}$ satisfy \eqref{eqn:special}, then
\[
e^{-\eta_{\pi\times\pi'}(x)} \leq \exp\Big(-\Cr{ZFR}\frac{\log x}{(m+m')\log(C(\pi)C(\pi'))+\sqrt{m(m+m')\Cr{ZFR}\log x}}\Big).
\]
\end{lemma}

\begin{proof}
By \eqref{eqn:eta_pi_def} with the change of variables $t \mapsto e^{u}$ and \cref{prop:ZFR_SD}, we have that
\[
\eta_{\pi\times\pi'}(x) \geq \inf_{u \geq 0} \phi_{x}(u),\qquad \phi_{x}(u) = \frac{\Cr{ZFR}\log x}{(m+m')\log(C(\pi)C(\pi'))+m(m+m')u} +u.
\]
Note that $\lim_{u\to\infty}\phi_{x}(u)= \infty$.  The equation $\frac{d}{du}\phi_x(u)=0$ has the unique positive solution
\[
u = u_0 := \Big(\frac{\Cr{ZFR}\log x}{m(m+m')}\Big)^{\frac{1}{2}}-\frac{\log(C(\pi)C(\pi'))}{m}.
\]
We have that $u_0 > 0$ if and only if $x> \exp(\frac{m+m'}{\Cr{ZFR}m}(\log(C(\pi)C(\pi')))^2)$, so
{\small\begin{align*}
\phi_x(u) &\geq \begin{cases}
	\phi_x(u_0)&\mbox{if $\displaystyle x> \exp\Big(\frac{m+m'}{\Cr{ZFR}m}(\log(C(\pi)C(\pi')))^2\Big)$,}\\
	\phi_x(0)&\mbox{if $\displaystyle 3\leq x\leq  \exp\Big(\frac{m+m'}{\Cr{ZFR}m}(\log(C(\pi)C(\pi')))^2\Big)$}
\end{cases}
\vspace{1mm}\\
&= \begin{cases}
\displaystyle 2\Big(\frac{\Cr{ZFR}\log x}{m(m+m')}\Big)^{\frac{1}{2}}-\frac{\log (C(\pi)C(\pi'))}{m}&\mbox{if $\displaystyle x> \exp\Big(\frac{m+m'}{\Cr{ZFR}m}(\log(C(\pi)C(\pi')))^2\Big)$,}\vspace{1mm}\\
\displaystyle \frac{\Cr{ZFR}\log x}{(m+m')\log(C(\pi)C(\pi'))}&\mbox{if $\displaystyle 3\leq x\leq  \exp\Big(\frac{m+m'}{\Cr{ZFR}m}(\log(C(\pi)C(\pi')))^2\Big)$}
\end{cases}\\
&\geq \min\Big\{\Big(\frac{\Cr{ZFR}\log x}{m(m+m')}\Big)^{\frac{1}{2}},\frac{\Cr{ZFR}\log x}{(m+m')\log(C(\pi)C(\pi'))}\Big\}.
\end{align*}}%
Since $\exp(-\min\{a,b\})\leq \exp(-\frac{a b}{a+b})$ when $a>0$ and $b>0$, the lemma follows.
\end{proof}
\begin{proof}[Proof of \cref{thm:PNT1,thm:PNT2}]
\cref{thm:PNT2} follows from \cref{thm:main} (with $A=2$) and \cref{Humphries-Thorner}.  We restrict the range of $x$ in order to absorb the factor of $(m'm)^{\Cr{main3}(m'm)^3}$ in the error term in \cref{thm:main}.  \cref{thm:PNT1} follows from \cref{thm:PNT2} by choosing $\pi'=\mathbbm{1}$.
\end{proof}

We perform similar analysis using Brumley's narrow zero-free region.
\begin{lemma}
\label{Brumley}
Let $\pi\in\mathfrak{F}_m$ and $\pi'\in\mathfrak{F}_{m'}$ satisfy $\pi'\neq\tilde{\pi}$.  Let $0<\epsilon<1$, and let $c_{m,m',\epsilon}$ be as in \eqref{eqn:ZFRBrumley}.  Let
\begin{equation}
\label{eqn:AB}
\mathcal{A} = c_{m,m',\epsilon}/(C(\pi) C(\pi'))^{(m+m')(1+\frac{\epsilon}{2})-1},\qquad \mathcal{B} = m'm (1-\tfrac{1}{m+m'}+\tfrac{\epsilon}{2}).
\end{equation}
If $x>\exp(3^{\mathcal{B}}/(\mathcal{A}\mathcal{B}))$, then
\[
e^{-\eta_{\pi\times\pi'}(x)} \leq (\mathcal{A}\mathcal{B}e\log x)^{-1/\mathcal{B}}.
\]
\end{lemma}

\begin{proof}
Let $\mathcal{A}$ and $\mathcal{B}$ be given by \eqref{eqn:AB}.  By \cref{prop:ZFR_Brumley} and \eqref{eqn:eta_pi_def}, we have that
\[
\eta_{\pi\times\pi'}(x) \geq \inf_{t \geq 3} \psi_{x}(t),\qquad \psi_{x}(t) = t^{-\mathcal{B}}\mathcal{A}\log x+\log t.
\]
Note that $\lim_{t\to\infty}\psi_{x}(t)= \infty$.  The equation $\frac{d}{dt} \psi_{x}(t) = 0$ has a unique positive solution $t_0 = (\mathcal{A}\mathcal{B} \log x)^{1/\mathcal{B}}$.  We have that $t_0>3$ if and only if $x>\exp(3^\mathcal{B}/(\mathcal{A}\mathcal{B}))$, in which case
\begin{align*}
\psi_x(t)\geq \begin{cases}
\phi_x(t_0)&\mbox{if $x>\exp(3^{\mathcal{B}}/(\mathcal{A}\mathcal{B}))$,}\\
\phi_x(3)&\mbox{if $3\leq x\leq \exp(3^{\mathcal{B}}/(\mathcal{A}\mathcal{B}))$}
\end{cases}= \begin{cases}
\frac{1+\log(\mathcal{A}\mathcal{B} \log x)}{\mathcal{B}}&\mbox{if $x>\exp(3^{\mathcal{B}}/(\mathcal{A}\mathcal{B}))$,}\\
\log 3+\frac{ \mathcal{A} \log x}{3^{\mathcal{B}}}&\mbox{if $3\leq x\leq \exp(3^{\mathcal{B}}/(\mathcal{A}\mathcal{B}))$.}
\end{cases}
\end{align*}
The lemma now follows.
\end{proof}
\begin{proof}[Proof of \cref{thm:PNT3}]
If $A\geq 2$ and $x\geq\exp(3^\mathcal{B} / (\mathcal{A}\mathcal{B}))$, then
\[
e^{-(1-\frac{1}{A})\eta_{\pi\times\pi'}(x)} \leq (\mathcal{A}\mathcal{B}e\log x)^{-(1-\frac{1}{A})/\mathcal{B}}
\]
by  \cref{Brumley}.  If $A=2(m+m')$, then
\[
(\mathcal{A}\mathcal{B}e\log x)^{-(1-\frac{1}{A})/\mathcal{B}}\ll_{m,m',\epsilon}(C(\pi)C(\pi'))^{\frac{1}{m}+\frac{1}{m'}-\frac{1}{2m'm}}(\log x)^{-\frac{2(m+m')-1}{m'm((2+\epsilon)(m+m')-2)}}
\]
If we let $\epsilon=(m+m')^{-2}$ and impose the condition $x\geq \exp((C(\pi)C(\pi'))^{2(m+m')^2})$, then
\[
(C(\pi)C(\pi'))^{\frac{1}{m}+\frac{1}{m'}-\frac{1}{2m'm}}(\log x)^{-\frac{2(m+m')-1}{m'm((2+\epsilon)(m+m')-2)}}\ll_{m,m'}(\log x)^{-\frac{1}{m'm}}.
\]
\cref{thm:PNT3} follows from this estimate, \cref{thm:main}, and \cref{lem:AC}.
\end{proof}

\bibliographystyle{abbrv}
\bibliography{bibliography.bib}

\def\cprime{$'$}
\begin{thebibliography}{10}

\bibitem{BFI2}
E.~Bombieri, J.~B. Friedlander, and H.~Iwaniec.
\newblock Primes in arithmetic progressions to large moduli. {II}.
\newblock {\em Math. Ann.}, 277(3):361--393, 1987.

\bibitem{BushnellHenniart1997}
C.~J. Bushnell and G.~Henniart.
\newblock An upper bound on conductors for pairs.
\newblock {\em J. Number Theory}, 65(2):183--196, 1997.

\bibitem{Hoffstein}
J.~Hoffstein and D.~Ramakrishnan.
\newblock Siegel zeros and cusp forms.
\newblock {\em Int. Math. Res. Not. IMRN}, 1995(6):279--308, 1995.

\bibitem{Humphries2019-2}
P.~Humphries.
\newblock Standard zero-free regions for {R}ankin-{S}elberg {$L$}-functions via
  sieve theory.
\newblock {\em Math. Z.}, 292(3--4):1105--1122, 2019.
\newblock With an appendix by Farrell Brumley.

\bibitem{HumphriesThorner2020}
P.~{Humphries} and J.~{Thorner}.
\newblock {Towards a $\mathrm{GL}_n$ variant of the Hoheisel phenomenon}.
\newblock {\em Trans. Amer. Math. Soc.}, 375:1801--1824, 2022.

\bibitem{IK}
H.~Iwaniec and E.~Kowalski.
\newblock {\em Analytic number theory}, volume~53 of {\em American Mathematical
  Society Colloquium Publications}.
\newblock American Mathematical Society, Providence, RI, 2004.

\bibitem{Lapid}
E.~Lapid.
\newblock On the {H}arish-{C}handra {S}chwartz space of {$G(F)\backslash
  G(\mathbb{A})$}.
\newblock In {\em Automorphic representations and {$L$}-functions}, volume~22
  of {\em Tata Inst. Fundam. Res. Stud. Math.}, pages 335--377. Tata Inst.
  Fund. Res., Mumbai, 2013.
\newblock With an appendix by Farrell Brumley.

\bibitem{Linnik}
U.~V. Linnik.
\newblock On the least prime in an arithmetic progression.
\newblock {\em Rec. Math. [Mat. Sbornik] N.S.}, 15(57):139--178,347--368, 1944.

\bibitem{MullerSpeh2004}
W.~M{\"u}ller and B.~Speh.
\newblock Absolute convergence of the spectral side of the {A}rthur trace
  formula for {${\rm GL}_n$}.
\newblock {\em Geom. Funct. Anal.}, 14(1):58--93, 2004.
\newblock With an appendix by E. M. Lapid.

\bibitem{SoundararajanThorner2019}
K.~Soundararajan and J.~Thorner.
\newblock Weak subconvexity without a {R}amanujan hypothesis.
\newblock {\em Duke Math. J.}, 168:1231--1268, 2019.
\newblock With an appendix by Farrell Brumley.

\bibitem{ThornerZaman2019}
J.~Thorner and A.~Zaman.
\newblock A unified and improved {C}hebotarev density theorem.
\newblock {\em Algebra Number Theory}, 13(5):1039--1068, 2019.

\end{thebibliography}

\end{document}